\documentclass[10pt]{amsart}
\usepackage{amsmath}
\usepackage{amsthm}
\usepackage{amssymb}
\usepackage{mathtools}
\usepackage{graphicx}
\usepackage[multiple]{footmisc}
\usepackage{multicol}
\usepackage{enumerate}
\usepackage{tikz}
\usetikzlibrary{arrows}

\usepackage[pdfauthor={Jonas Reitz and Kameryn J Williams},
    pdftitle={Inner Mantles And Iterated HOD},
    hidelinks 
]{hyperref}

\usepackage[backend=bibtex,style=alphabetic,maxbibnames=15,maxcitenames=6,dateabbrev=false]{biblatex}
\renewcommand{\UrlFont}{\sffamily\small} % makes url text smaller (used only in bibliography?)
\renewbibmacro{in:}{\ifentrytype{article}{}{\printtext{\bibstring{in}\intitlepunct}}} % removes "In:" in article format.
\DeclareFieldFormat{url}{{\UrlFont\url{#1}}} % removes redundant "URL:" text, and just gives the url.
\DeclareFieldFormat{urldate}{% improves setting of urldate
  (version \thefield{urlday}\addspace%
  \mkbibmonth{\thefield{urlmonth}}\addspace%
  \thefield{urlyear}\isdot)}
\DeclareFieldFormat{eprint:arxiv}{
  \ifhyperref
    {\href{http://arxiv.org/abs/#1}{%
        arXiv\addcolon\nolinkurl{#1}}\iffieldundef{eprintclass}{}{\UrlFont{\mkbibbrackets{\thefield{eprintclass}}}}}
    {arXiv\addcolon\nolinkurl{#1}\iffieldundef{eprintclass}{}{\UrlFont{\mkbibbrackets{\thefield{eprintclass}}}}}}

\bibliography{HamkinsBiblio,LogicBiblio,local}
\theoremstyle{plain}
\ifdefined\theorem \else \newtheorem{theorem}{Theorem} \fi
\ifdefined\maintheorem \else \newtheorem{maintheorem}[theorem]{Main Theorem} \fi
\ifdefined\lemma \else \newtheorem{lemma}[theorem]{Lemma} \fi
\ifdefined\lemmaschema \else \newtheorem{lemmaschema}[theorem]{Lemma Schema} \fi
\ifdefined\proposition \else \newtheorem{proposition}[theorem]{Proposition} \fi
\ifdefined\corollary \else \newtheorem{corollary}[theorem]{Corollary} \fi
\ifdefined\fact \else \newtheorem{fact}[theorem]{Fact} \fi
\ifdefined\problem \else \newtheorem{problem}[theorem]{Problem} \fi
\ifdefined\conjecture \else \newtheorem{conjecture}[theorem]{Conjecture} \fi
\ifdefined\question \else \newtheorem{question}[theorem]{Question} \fi
\ifdefined\observation \else \newtheorem{observation}[theorem]{Observation} \fi
\newtheorem*{theorem*}{Theorem}
\newtheorem*{lemma*}{Lemma}
\newtheorem*{proposition*}{Proposition}

\theoremstyle{definition}
\ifdefined\definition \else \newtheorem{definition}[theorem]{Definition} \fi
\theoremstyle{remark}
\ifdefined\remark \else \newtheorem{remark}[theorem]{Remark} \fi
\ifdefined\remarks \else \newtheorem{remarks}[theorem]{Remarks} \fi
\ifdefined\example \else \newtheorem{example}[theorem]{Example} \fi
\ifdefined\claim \else \newtheorem{claim}[theorem]{Claim} \fi
\ifdefined\acknowledgment \else \newtheorem*{acknowledgment}{Acknowledgment} \fi
\ifdefined\dedication \else \newtheorem*{dedication}{Dedicated} \fi
\ifdefined\case \else \newtheorem*{case}{Case} \fi

\newcounter{my_enumerate_counter}

\newcommand\comment[1]{}

\newcommand\Pcal{\mathcal{P}}

\newcommand\Abb{\mathbb{A}}
\renewcommand\Bbb{\mathbb{B}}
\newcommand\Cbb{\mathbb{C}}

\newcommand\Mbb{\mathbb{M}}
\newcommand\Nbb{\mathbb{N}}
\newcommand\Obb{\mathbb{O}}
\newcommand\Pbb{\mathbb{P}}
\newcommand\Qbb{\mathbb{Q}}
\newcommand\Rbb{\mathbb{R}}
\newcommand\Sbb{\mathbb{S}}

\newcommand{\dom}{\operatorname{dom}}

\ifdefined\alt \else
  \newcommand{\alt}{\operatorname{alt}}
\fi

\newcommand{\one}{\mathbf{1}}
\newcommand{\zero}{\mathbf{0}}

\newcommand{\forces}{\Vdash}

\newcommand\axiom{\mathsf}

\newcommand\GCH{\axiom{GCH}}

\newcommand\ZFC{\axiom{ZFC}}

\newcommand\ETR{\axiom{ETR}}

\newcommand\class{\mathrm}

\newcommand\HOD{\class{HOD}}
\newcommand\Ord{\class{Ord}}

\newcommand\Add{\class{Add}}

\newcommand\tail{\textrm{tail}}

\newcommand{\seq}[1]{\left\langle #1 \right\rangle}

\renewcommand{\epsilon}{\varepsilon}

\newcommand\mand{\textrm{ and }}

\newcommand\card[1]{\left\lvert #1 \right\rvert}

\newcommand\rest{\upharpoonright}

\newcommand\powerset{\Pcal}
\newcommand\rank{\operatorname{rank}}

\newcommand\Con{\operatorname{Con}}

\ifdefined\Form \else
  \newcommand\Form{\axiom{Form}}
\fi

\renewcommand{\phi}{\varphi}

\title{Inner Mantles and Iterated HOD}

\author{Jonas Reitz}
\address[Jonas Reitz]{New York City College of Technology of The City University of New York, Mathematics, 300 Jay Street, Brooklyn, NY 11201}
\email{jreitz@citytech.cuny.edu}

\DeclareRobustCommand{\okina}{%
  \raisebox{\dimexpr\fontcharht\font`A-\height}{%
    \scalebox{0.8}{`}%
  }%
}
\author{Kameryn J Williams}
\address[Kameryn J Williams]{
University of Hawai\okina{}i at M\=anoa \\
Department of Mathematics \\
2565 McCarthy Mall, Keller 401A \\
Honolulu, HI  96822\\
USA}
\email{kamerynw@hawaii.edu}
\urladdr{http://kamerynjw.net}

\thanks{We thank the referee for their helpful comments.}

\begin{document}

\begin{abstract}
We present a class forcing notion $\Mbb(\eta)$, uniformly definable for ordinals $\eta$, which forces the ground model to be the $\eta$-th inner mantle of the extension, in which the sequence of inner mantles has length at least $\eta$. This answers a conjecture of Fuchs, Hamkins, and Reitz \cite{FuchsHamkinsReitz2015:Set-theoreticGeology} in the positive. We also show that $\Mbb(\eta)$ forces the ground model to be the $\eta$-th iterated $\HOD$ of the extension, where the sequence of iterated $\HOD$s has length at least $\eta$. We conclude by showing that the lengths of the sequences of inner mantles and of iterated $\HOD$s can be separated to be any two ordinals you please. 
\end{abstract}

\maketitle

\section{Introduction and history}

A \emph{ground} is an inner model $W$ so that there is some forcing notion $\Pbb \in W$ and $G \in V$ generic for $\Pbb$ over $W$ so that $W[G] = V$. It is a remarkable result, due independently to Laver \cite{Laver2007:CertainVeryLargeCardinalsNotCreated} and Woodin \cite{Woodin2004:CHMultiverseOmegaConjecture,Woodin2004:RecentDevelopmentsOnCH}, that the grounds are uniformly definable, which allows for the grounds to be quantified over in a first-order context. The \emph{mantle}, introduced in \cite{FuchsHamkinsReitz2015:Set-theoreticGeology}, is the intersection of all the grounds. It follows from work of Usuba \cite{usuba2017} that the mantle is preserved by set forcing and is the largest set-forcing-invariant inner model.

Fuchs, Hamkins, and Reitz \cite{FuchsHamkinsReitz2015:Set-theoreticGeology} produced a class forcing notion which forces the ground model to be the mantle of the forcing extension. In that same article they conjectured that every model of $\ZFC$ is the $\eta$-th inner mantle of another model, in which the sequence of inner mantles does not stabilize before $\eta$. The first main theorem of our paper answers their conjecture in the positive.

\begin{maintheorem}
Let $\eta$ be an ordinal. There is a class forcing notion $\Mbb(\eta)$, uniformly definable in $\eta$, so that forcing with $\Mbb(\eta)$ forces the ground model to be the $\eta$-th inner mantle of the extension, in which the sequence of inner mantles does not stabilize before $\eta$. Indeed, forcing with $\Mbb(\eta)$ forces the ground model to be the $\eta$-th iterated $\HOD$ of the extension, with the sequence of iterated $\HOD$s not stabilizing before $\eta$.
\end{maintheorem}

As our second main theorem we show that the lengths of these sequences can be separated.

\begin{maintheorem}
Let $\zeta$ and $\eta$ be ordinals.
\begin{itemize}
\item There is a class forcing, uniformly definable in $\zeta$ and $\eta$, which forces the sequence of inner mantles to have length $\zeta + \eta$ and the sequence of iterated $\HOD$s to have length $\zeta$.
\item There is a class forcing, uniformly definable in $\zeta$ and $\eta$, which forces the sequence of inner mantles to have length $\zeta$ and the sequence of iterated $\HOD$s to have length $\zeta + \eta$.
\end{itemize}
\end{maintheorem}

\begin{definition}
The sequence of \emph{inner mantles} can be defined as follows. The zeroth inner mantle $M^0$ is simply the universe $V$. Given the $\alpha$-th inner mantle $M^\alpha$, the $(\alpha+1)$-th mantle is $M^{\alpha+1} = M^{M^\alpha}$, the mantle of the $\alpha$-th inner mantle. For limit stages $\lambda$, the $\lambda$-th mantle is $M^\lambda = \bigcap_{\alpha < \lambda} M^\alpha$, the intersection of the previous inner mantles. Say that the sequence of inner mantles \emph{stabilizes} if there is some $\alpha$ so that $M^\alpha = M^{\alpha+1}$. For $\alpha$ least such that this happens we say that the sequence \emph{stabilizes at $\alpha$}. If the sequence stabilizes at $\alpha$ we also say that it has \emph{length $\alpha$}, and if it does not stabilize by $\alpha$ we also say that it has \emph{length at least $\alpha$}.
\end{definition}

Observe that if we allow the sequence of inner mantles to stabilize before $\eta$, it is trivial to get a model which is the $\eta$-th inner mantle of some outer model. This can be done by taking any model of the Ground Axiom \cite{Reitz2007:TheGroundAxiom}, which asserts that $V = M$. In this case, for any $\eta$ the $\eta$-th inner mantle is simply $V$.

Formally, we will formulate our work in a weak second-order set theory, with proper classes as actual objects rather than mere syntactic sugar. Namely, we will take a version of G\"odel--Bernays set theory with predicative comprehension, comprehension for formulae which only quantify over sets, as our background theory. We will assume that the axiom of choice holds for sets, but will not assume any form of global choice.\footnote{A precise axiomatization can be found in the appendix of \cite{Reitz2006:Dissertation}.}
Given any model of $\ZFC$, attaching its definable classes gives a model of this theory, but there are many models of this theory which have undefinable classes.\footnote{For example, if $\kappa$ is inaccessible then take $V_\kappa$ along with its powerset for the classes. A downward L\"owenheim--Skolem argument then gives countable models with undefinable classes.} 
So our approach is  more general than working in $\ZFC$ with definable classes, definable class forcing notions, and so forth.

With this context in mind, the above definition of the inner mantles is formalized as asserting the existence of certain sequences of classes. Namely, we say that the $\alpha$-th inner mantle exists if there is a sequence $\vec M$ of classes of length $\alpha+1$ which satisfies the recursive definition of the sequence of inner mantles. It is obvious that for finite standard $n$, the $n$th inner mantle always exists. We can inductively in the metatheory define the $n+1$ length sequence. But the complexity of the definitions increase as $n$ does, so it is far from clear that the $\omega$-th mantle always exists.\footnote{For the analogous question about the sequence of iterated $\HOD$s, McAloon \cite{McAloon1971:ConsistencyResultsAboutOrdinalDefinability} showed that there are models with $\ZFC$ whose $\omega$-th iterated $\HOD$ does not exist.}
We will not address the general question of when the $\eta$-th inner mantle exists. 
In the models we produce by forcing with $\Mbb(\eta)$ we will always have that the $\eta$-th inner mantle exists, so that question will not affect us.

The reader may compare our results with previous results about iterated $\HOD$, by McAloon \cite{McAloon1971:ConsistencyResultsAboutOrdinalDefinability}, Zadro\.zny \cite{Zadrozny83:IteratingOrdinalDefinability}, Jech \cite{Jech:ForcingWithTreesandOrdinalDefinability} and others. 

\begin{definition} \label{definition.iteratedHOD}
The \emph{sequence of iterated $\HOD$s} is defined similarly to the sequence of inner mantles, but taking the $\HOD$ at each stage rather than the mantle. Namely, $\HOD^0 = V$, $\HOD^{\alpha+1} = \HOD^{\HOD^\alpha}$, and $\HOD^\lambda = \bigcap_{\alpha < \lambda} \HOD^\alpha$ for limit $\lambda$. %% redefine stabilizing/length here? --KW
\end{definition}

These results on iterated $\HOD$ were an inspiration for much of our current work. The strongest result in that context is due to Zadro\.zny \cite{Zadrozny83:IteratingOrdinalDefinability}, who showed that any model of $\ZFC$ is the $\eta$-th iterated $\HOD$ of some generic extension for any ordinal $\eta$ or for $\eta = \Ord$, where the sequence does not stabilize before $\eta$. So the new content of our first main theorem is that the analogous fact is true for inner mantles.

The main open question left from our work is whether we can take the inner mantle sequence to $\Ord$ and beyond. The definition of the inner mantle can be carried out along any class well-order, even one longer than $\Ord$. Given a class well-order $\Gamma$, is there a class forcing notion $\Mbb(\Gamma)$ so that forcing with $\Mbb(\Gamma)$ makes the ground model the $\Gamma$-th inner mantle of the extension, where the sequence of inner mantles does not stabilize before $\Gamma$? Can we get models where the sequence of inner mantles does not stabilize at any any class well-order?

In section \ref{sec:meta} we will give the definition of $\Mbb(\eta)$ and prove some basic lemmata about it. The first main theorem is proved in section \ref{sec:main-theorem}, and the second main theorem is proved in section \ref{sec:different-lengths}.

\section{The definition of the forcing} \label{sec:meta}

The basic building block we will employ is the forcing which makes the ground model equal to the mantle of the extension ($V=M^{V[G]}$). This forcing is covered in detail in \cite{FuchsHamkinsReitz2015:Set-theoreticGeology} and will be described further below, but for now it suffices to note that it is a class product that performs coding on a class $R$ of regular cardinals -- that is, at each cardinal $\alpha \in R$, the forcing chooses whether to make the $\GCH$ hold or fail at $\alpha$, and forces accordingly.  The class of coding points $R$ should be chosen to satisfy two properties:  first, that the coding at different cardinals in $R$ should not interfere with one another, and second that $R$ be a definable class in the extension. This is straightforward when the $\GCH$ holds in $V$, and slightly more complicated in the general case.  For ease of presentation, we will assume for now that an appropriate class of coding points $R$ has been determined, leaving the definition of $R$ for after the definition of the forcing.

Note that simply iterating the $V=M^{V[G]}$ forcing a finite number of times gives rise to a model $V[G_1*G_2*\dots*G_n]$ with a finite sequence of mantles leading back to the original model $V$.  Unfortunately, the correspondence between the stages of the forcing iteration and the successive mantles is reversed, so that the first mantle in the extension is obtained by removing the generic corresponding to the last stage of the forcing iteration, i.e.  $M^{V[G_1*G_2*\dots*G_n]} = V[G_1*G_2*\dots*G_{n-1}]$. 
Thus, to obtain a transfinite sequence of inner mantles is not so simple as to iterate this forcing $\eta$ many times. We will instead need to perform an iteration along a reversed well-order, with each successive mantle stripping another generic from the end of the iteration, taking intersections at limits.

Fix an ordinal $\eta$. Intuitively, we would like to define an iteration of class forcing along the non-well-founded order $\eta^\star$.  In the resulting extension, the mantle sequence will correspond exactly to the intermediate generic extensions given by the initial segments of the iteration.  This forcing will be presented in two ways, and the interplay between these views will allow us to analyze the behavior of the mantles in the resulting extension.  Initially we will present the definition and analysis of the forcing as a traditional iteration of set forcing of order type $\Ord$, and then, with a change of perspective, we will describe how it can be viewed as a non-well-founded iteration of class forcing of order type $\eta^\star$, where $\eta^\star$ is the reversed order of $\eta$.

For later arguments we will need our forcing to be $\mathord\le\eta^+$-closed, where $\eta^+$ is the least cardinal $> \eta$. So we will assume that $R$ consists only of cardinals $>\eta^+$. Observe that this assumption can be made without loss, as if $R$ is an appropriate class of cardinals then any tail of $R$ is also appropriate for our forcing. We will partition $R$ into $\eta$ many cofinal classes: $R_i$ for $i < \eta$ consists of the $(\eta \cdot \zeta + i)$-th elements of $R$, arranged in ascending order, for some $\zeta$. That is, a cardinal in $R$ is in $R_i$ if the index of its position in $R$ is equivalent to $i$ modulo $\eta$. For a cardinal $\alpha \in R$ let $i(\alpha)$, the index of $\alpha$, be the unique $i < \eta$ so that $\alpha \in R_i$. Let $R_{>i} = \bigcup_{j > i} R_j$ and $R_{\ge i} = \bigcup_{j \ge i} R_j$.

\begin{definition} \label{def:main-def}
$\Mbb(\eta)$ is the class forcing whose conditions are set-sized functions $p$ with domain an initial segment of $R$ so that for all $\alpha \in \dom p$ we have $p(\alpha)$ is an $\Mbb(\eta) \rest (R_{>i(\alpha)} \cap \alpha)$-name for a condition in $\Add(\alpha,(2^{<\alpha})^{++}) \oplus \Add(\alpha^+,1)$, the forcing to generically make $\GCH$ either fail or hold at $\alpha$.\footnote{In the case where the ground model $V$ satisfies $\GCH$, it suffices to use the simpler forcing $\Add(\alpha,\alpha^{++}) \oplus \zero$ at stage $\alpha$.}
Conditions $p$ can have arbitrary support.
Given $p,q \in \Mbb(\eta)$, we say $q \le p$ if $\dom(q) \supseteq \dom(p)$ and for all $\alpha \in \dom(p)$ we have that $p \rest (R_{>i(\alpha)} \cap \alpha)$ forces over $\Mbb(\eta) \rest (R_{>i(\alpha)} \cap \alpha)$ that $q(\alpha) \le p(\alpha)$, where $\Mbb(\eta) \rest A$ means the collection of the restrictions of $p \in \Mbb(\eta)$ to $A$.
\end{definition}

This forcing $\Mbb(\eta)$ is a linear iteration of length $\Ord$. However, the reader should be aware that the forcing at stage $\alpha$ is defined in a certain inner model of the forcing extension by $\Mbb(\eta) \rest \alpha$, the restriction to stages below $\alpha$, rather than in the full extension. Accordingly, the analysis requires a little more care than just citing well-known results about iterations. See below for further discussion. 

Having defined $\Mbb(\eta)$, let us now consider the coding points $R$. If $\GCH$ holds then it is clear that we can pick an appropriate class $R$ of coding points---take the regular cardinals. But if $\GCH$ fails then we need a little more care. To motivate our choice of $R$, let us consider which cardinals are preserved by forcing with $\Add(\alpha,(2^{<\alpha})^{++})$. This forcing is $\mathord<\alpha$-closed, so it preserves cardinals $\leq\alpha$. And it has the $(2^{<\alpha})^{+++}$-chain condition, so it preserves cardinals $>(2^{<\alpha})^{++}$ (the alternative forcing in the lottery sum at stage $\alpha$, $\Add(\alpha^+,1)$, is even better behaved, preserving cardinals $\leq \alpha$ and $> 2^{<\alpha}$). So if we want our coding points to not interfere with each other then we need the next coding point after $\alpha$ to be at least $(2^{<\alpha})^{++}$. But we also want that $R$ is preserved by the forcing, so we will space things out further to make it easy to see that this is the case. This extra spacing out is harmless, because real estate is cheap when you have $\Ord$ much room. Note that forcing with $\Add(\alpha,(2^{<\alpha})^{++})$ or $\Add(\alpha^+,1)$ only affects a finite interval of the beth-numbers, and therefore preserves the limit elements of the beth-numbers. So we can use them as guideposts for our coding points. Specifically, we will take $R$ to consist of those successor cardinals $\delta^+$ where $\delta$ is a strong limit but not a limit of strong limits.  That is, we define
\[
R = \{ (\beth_{\omega \cdot (\beta + 1)})^+ : \beta \in \Ord \mand (\beth_{\omega \cdot (\beta + 1)})^+ > \eta^+ \}.
\]

The remainder of this section will provide an analysis of the forcing $\Mbb(\eta)$ with coding points $R$, establishing the basic facts that will be used in the proof of the main theorem---namely, that $\Mbb(\eta)$ preserves $\ZFC$, preserves $R$, and that in the extension the $\GCH$ holds at each $\alpha\in R$ exactly according to the winner of the lottery at stage $\alpha$.  Let us begin by considering $\Mbb(\eta)$ as an iteration of set-sized forcing indexed by the class $R$. In the following section, in the proof of the main theorem, we will shift our view to consider $\Mbb(\eta)$ as an iteration of class forcing of order type $\eta^\star$.

Forcing in $\Mbb(\eta)$ at stage $\alpha \in R$ is defined as the partial order $\Add(\alpha,(2^{<\alpha})^{++}) \oplus \Add(\alpha^+,1)$, but rather than taking this definition in the extension up to $\alpha$---namely $V^{\Mbb(\eta) \rest R\cap \alpha}$---instead we restrict the definition to the inner model $V^{\Mbb(\eta) \rest (R_{>i(\alpha)} \cap \alpha)}$ obtained by taking generics at only those stages $\beta < \alpha$ with index $i(\beta) > i(\alpha)$.  This complicates the analysis as, for example, the forcing at stage $\alpha$ need not be $\mathord<\alpha$-closed, and so the standard iteration arguments break down.  Our strategy will be to embed $\Mbb(\eta)$ into a more familiar and better-behaved forcing, a class product of ground-model set forcing, in such a way that the extension by $\Mbb(\eta)$ inherits many of the properties of the extension by the product. 

To come up with an appropriate ground-model forcing that will effectively absorb $\Mbb(\eta)$, we utilize recent work of the first author \cite{Reitz2018:CohenForcing}.  $\Mbb(\eta)$ is an example of a broad class of iterations in which forcing at stage $\alpha$ consists of Cohen forcing at $\alpha$ as defined in some inner model.  As it turns out, such an iteration can be embedded into a ground-model product by simply replacing the stage $\alpha$ forcing with the corresponding ground model forcing.

\begin{definition}\label{definition.generalizedcoheniteration}Suppose $\Pbb = \left<\Pbb_\alpha, \dot\Qbb_\alpha \mid \alpha \in R\right>$ is an iteration along a class $R$ of regular cardinals with either Easton support or full (set) support.  Then $\Pbb$ is a \emph{generalized Cohen iteration} provided, for each $\alpha \in R$,
\begin{enumerate}
    \item $\dot\Qbb_\alpha$ is a full $\Pbb_\alpha$-name for a partial order and $\Pbb_\alpha \forces \dot\Qbb_\alpha = \dot\Add(\check\alpha, \check{\lambda_\alpha})^{V^{\Rbb_\alpha}}$, where
    \item $\Rbb_\alpha$ is a (possibly trivial) complete suborder of $\Pbb_\alpha$
    
    \item $\displaystyle \Rbb_\alpha \forces \left|\check\Add(\check\alpha,1)\right| = \left|\dot\Add(\check\alpha, 1)^{V^{\Rbb_\alpha}}\right|$, \\ 
    $\displaystyle \Rbb_\alpha \forces \left|\check\Add(\check\alpha,\check{\lambda_\alpha})\right| = \left|\dot\Add(\check\alpha, \check{\lambda_\alpha})^{V^{\Rbb_\alpha}}\right|$, and \\
    $\displaystyle \Rbb_\alpha \forces \check{V}\subseteq V^{\Rbb_\alpha}$ satisfies the $\alpha$-cover property for subsets of $\lambda_\alpha$ (that is, any $A\subset \lambda_\alpha$ from $V^{\Rbb_\alpha}$ of size $\mathord<\alpha$ in that model is contained in some $B$ from $V$ of size $\mathord<\alpha$).
\end{enumerate}
\end{definition}

\begin{theorem}\label{thm.generalizedcoheniterations}Suppose $\Pbb$ is a generalized Cohen iteration.  Then there is a projection map $\pi:\Pi_{\alpha\in R} \Add(\alpha,\lambda_\alpha)\to\Pbb$, where  $\Pi_{\alpha\in R} \Add(\alpha,\lambda_\alpha)$ is the ground model class product (with the same supports as $\Pbb$).  Furthermore,
\begin{enumerate}
    \item $\pi$ restricted to an initial segment of $R$ maps to the corresponding initial segment of $\Pbb$, that is 
    $\pi\rest\delta : \Pi_{\alpha \in R\cap\delta} \Add(\alpha,\lambda_\alpha) \to \left<\Pbb_\alpha, \dot\Qbb_\alpha \mid \alpha \in R\cap\delta\right>$ is a projection map, and for $p\in\Pbb$ we have $\pi(p\rest\delta)=\pi(p)\rest\delta$,
    \item for each $\alpha \in R$ we can factor $\Pbb \cong \Pbb_\alpha * \Pbb^\tail$ where $\Pbb^\tail$ is a $\Pbb_\alpha$-name for the tail forcing and $\Pbb_\alpha \forces \Pbb^\tail$ is $\mathord<\alpha$-distributive.
\end{enumerate}
\end{theorem}

Projection maps are standard notion \cite{Jech:SetTheory3rdEdition},  dual to that of complete embeddings, which allows a generic subset of the domain to generate a generic subset of the target.

\begin{proof}
A complete proof of the theorem appears in \cite{Reitz2018:CohenForcing}.  Here, we provide a brief sketch of the method.  The projection map is defined inductively level-by-level through repeated application of the following: 

\begin{lemma}\label{lem.cohenforcingprojectionlemma}
If $\alpha$ is a regular cardinal in $V$, $\lambda$ is a cardinal in $V$, and $W\subset V$ is an inner model satisfying:
\begin{enumerate}
    \item $|\Add(\alpha,1)^W|=|\Add(\alpha,1)^V|$, 
    \item $|\Add(\alpha,\lambda)^W|=|\Add(\alpha,\lambda)^V|$, and 
    \item $W\subset V$ satisfies the $\alpha$-cover property for subsets of $\lambda$,
\end{enumerate}
then there is a projection map $\pi:\Add(\kappa,\lambda)^W\to \Add(\kappa,\lambda)^V$.
\end{lemma}

According to the lemma, we have a projection map from the ground model forcing $\Add(\alpha,\lambda_\alpha)^V$ to the forcing $\Add(\alpha,\lambda)^{V^{\Rbb_\alpha}}$ at stage $\alpha$ provided certain  restrictions are met---in particular, provided the posets $\Add(\alpha,1)$ and $\Add(\alpha,\lambda)$ of the extension $V^{\Rbb_\alpha}$ have the same sizes as those of the ground model, and the models $V\subseteq V^{\Rbb_\alpha}$ satisfy the $\alpha$-cover property for subsets of $\lambda_\alpha$.  Use of full names in the  iteration allows names for projection maps at each stage to be combined into a projection map in $V$ from the product of ground model posets
$\Add(\alpha,\lambda_\alpha)$ to $\Pbb$.  Note that standard analysis of the product shows that the tail forcing (the stages beyond $\alpha$) is $\mathord<\alpha$-distributive and so adds no $\mathord<\alpha$-sequences over the ground model.  The level-by-level nature of the projection shows that tail of $\Pbb$ (beyond $\Pbb_\alpha$) therefore cannot add $\mathord<\alpha$-sequences over the ground model, and so must also be $\mathord<\alpha$-distributive.  This establishes that $\Pbb$ is a progressively distributive iteration.
\end{proof}

We next consider the application of Theorem \ref{thm.generalizedcoheniterations} to $\Mbb(\eta)$.  Unfortunately, $\Mbb(\eta)$ does not quite have the form of a generalized Cohen iteration (forcing at stage $\alpha$ is not simply Cohen forcing but the lottery sum of two Cohen partial orders).  However, $\Mbb(\eta)$ is densely equal to a generalized Cohen iteration.  In particular, since full set support is used, for each $\alpha$ there is a dense set of conditions that decide which selection the lottery will take at every stage below $\alpha$.  We can make this observation more concrete by factoring $\Mbb(\eta)$ as $\Mbb(\eta)=\Add(R,1)*\Mbb(\eta)_{\dot A}$.  By the first factor $\Add(R,1)$ we mean the forcing with conditions $p:R\cap \alpha \to 2$ for each $\alpha$, ordered by extension, which adds a generic Cohen class $A\subset R$.  Since full set support is used, this is equivalent to forcing with $2^{<\Ord}$, which is $\alpha$-closed for every $\alpha$ and therefore does not add any sets.  The class $A$ determines which option the lottery takes at each stage.  The second factor, $\Mbb(\eta)_{\dot A}$, is an $\Add(R,1)$-name for the $\Ord$-length iteration of Cohen forcing where the particular Cohen poset used at stage $\alpha$ is determined by whether $\alpha\in A$, that is, by the lottery choice made at $\alpha$.
%Over $V[A]$, our iteration $\Mbb(\eta)$ is reduced to a simple $\Ord$-length iteration of Cohen forcing, let us call it $\Mbb(\eta)_A$, where at stage $\alpha$ we force with either $\Add(\alpha,(2^{<\alpha})^{++})^{V^{\Rbb_\alpha}}$ or $\Add(\alpha^+,1)^{V^{\Rbb_\alpha}}$ depending on the lottery choice already determined by $A$ (for an appropriate inner model $V^{\Rbb_\alpha}$).  
In $V[A]$ the name $\Mbb(\eta)_{\dot A}$ is resolved into an actual iteration $\Mbb(\eta)_A$.  It remains to show that this iteration is a generalized Cohen iteration (Definition \ref{definition.generalizedcoheniteration}).

Fix $\alpha \in R$, and let $\Rbb_\alpha = \Mbb(\eta)_A \rest (R_{>i(\alpha)} \cap \alpha)$.  The spacing between elements of $R$ allows us to establish inductively that the forcing $\Mbb(\eta)_A\rest\alpha$ up to $\alpha$ has size and chain condition strictly $<\delta$, where $\alpha=\delta^+$. It follows that $\alpha$ remains a cardinal in $V^{\Mbb(\eta)_A\rest\alpha}$, and that the sizes of the posets $\Add(\alpha,1)$ and $\Add(\alpha,(2^{<\alpha})^{++})$  are the same as they were in $V$.  As $V^{\Rbb_\alpha}$ is an intermediate model between $V$ and $V^{\Mbb(\eta)_A\rest \alpha}$, we conclude
\begin{align*}
&\Rbb_\alpha \forces \card{\check\Add(\check\alpha,1)} = \card{\dot\Add(\check\alpha, 1)^{V^{\Rbb_\alpha}}},\\
 &\Rbb_\alpha \forces \left|\check\Add(\check\alpha,\check{(2^{<\alpha})^{++}})\right| = \left|\dot\Add(\check\alpha,\check{(2^{<\alpha})^{++}})^{V^{\Rbb_\alpha}}\right|, \mand\\
 & \Rbb_\alpha \forces \check{V}\subseteq V^{\Rbb_\alpha} \text{ satisfies the } \alpha\text{-cover property for subsets of }\check{(2^{<\alpha})^{++}}.
\end{align*}
Note that the cover property in the final line follows from the chain condition of $\Rbb_\alpha$, a complete suborder of $\Mbb(\eta)_A\rest\alpha$.  Thus, in $V[A]$, we have established that $\Mbb(\eta)_A$ satisfies the hypotheses of Theorem \ref{thm.generalizedcoheniterations}.

Moreover note that the same argument applies to $\Mbb(\eta)_A \rest R_{\ge i}$, the restriction of $\Mbb(\eta)_A$ to only those coordinates $\alpha$ with index $\ge i$. So $\Mbb(\eta)_A \rest R_{\ge i}$ also satisfies the conditions for Theorem \ref{thm.generalizedcoheniterations}.

Many traditional iterations have a factoring property in which the tail forcing can be made to satisfy an arbitrary degree of closure by choosing an appropriate initial factor (call these \emph{progressively closed iterations} as in \cite{Reitz2006:Dissertation}).  Weakening this property from closure to distributivity in the tails, as in the case of Theorem \ref{thm.generalizedcoheniterations}, still suffices to preserve $\ZFC$---let us refer to such constructions as \emph{progressively distributive iterations}.

\begin{definition}\label{def:progressively-distributive-iteration}
$\Pbb$ is a \emph{progressively distributive iteration} if and only if for arbitrarily large regular $\alpha$ we can factor $\Pbb = \Pbb_\alpha * \dot{\Pbb}^{\tail}$ where $\Pbb_\alpha$ is a set and $\Pbb_\alpha \forces \dot{\Pbb}^{\tail}$ is $\mathord<\alpha$-distributive.
\end{definition}

\begin{lemma}\label{lem:progressively-distributive-iterations-are-tame}
If $\Pbb$ is a progressively distributive iteration then $\Pbb$ preserves $\ZFC$.
\end{lemma}

\begin{proof}
Showing that progressively distributive iterations preserve $\ZFC$ is a straightforward modification of the proof for progressively closed iterations \cite{Reitz2006:Dissertation}, with powerset and replacement the only axioms requiring discussion.  To see that the extension $V[G]$ by a progressively distributive iteration $\Pbb$ preserves powerset, we fix a set $a \in V[G]$ and, choosing appropriate $\alpha > \card{a}$, we factor $\Pbb = \Pbb_\alpha * \dot{\Pbb}^{\tail}$.  Distributivity of the tail forcing shows that both $a$ and all subsets of $a$ from $V[G]$ lie in the set forcing extension $V[G_\alpha]$ by $\Pbb_\alpha$, which is a model of $\ZFC$ and so contains the powerset $\powerset(a)$.  Thus $\powerset(a) \in V[G]$.  

For the replacement axiom, fix a name $\dot F$ for a class function in the extension and fix a set $a$ in the extension. By progressively closed distributivity we can factor $\Pbb$ as above so that the tail forcing is $\mathord\le \card{a}$-distributive. Now note that for each $x \in a$ the class of conditions deciding the value of $\dot F(\check x)$ is dense in the tail forcing. So by $\mathord\le \card{a}$-distributivity of $\Pbb^\tail$ we can obtain a single dense set of conditions which decide all values of $F \rest a$. Pick such condition $q \in G$ and consider the name $\dot b = \{ (q, \dot y) : q \forces_{\Pbb^\tail} \dot F(\check x) = \dot y$ for some $x \in a \}$. Then $\dot b$ is a set name and gives $F''a$ when interpreted by the generic, so $F''a$ is a set in $V[G]$, as desired. Thus progressively distributive iterations preserve $\ZFC$.
\end{proof}

% QUESTION: Do we need to add any clarification to the following paragraph, given the modifications above (e.g. working in V[A] instead of V)?

For the convenience of the reader, we summarize here some properties about $\Mbb(\eta)$. It is $\mathord\le\eta^+$-closed. Forcing with $\Mbb(\eta)$ preserves $R$. For $\alpha \in R$ it factors into $\Pbb_{<\alpha} * \dot\Qbb_\alpha * \dot\Pbb^\tail$ where $\Pbb_{<\alpha}$ has cardinality $<\alpha$ and $\Pbb_{<\alpha} * \dot\Qbb_\alpha$ forces that $\dot\Pbb^\tail$ is $\mathord<\alpha$-distributive. And analogous facts hold for $\Mbb(\eta) \rest R_{\ge i}$ for each $i \le \eta$, the restriction of $\Mbb(\eta)$ to only those coordinates with index $\ge i$.\footnote{To clarify, $\Mbb(\eta) \rest R_{\ge \eta}$ is empty, since every condition has index $< \eta$. What we mean in this case is the trivial forcing $\{ \one_{\Mbb(\eta)} \}$. Cf.\ the remarks at the beginning of the proof of Theorem \ref{thm:main-thm} about canonically embedding $\Mbb(\eta) \rest R_{\ge i}$ into $\Mbb(\eta)$.}

\section{Forcing the ground model to be an inner mantle} \label{sec:main-theorem}

In this section we prove the first main theorem of this article, that forcing with $\Mbb(\eta)$ makes the ground model the $\eta$-th inner mantle and $\eta$-th iterated $\HOD$ of the extension. We start by considering the sequence of inner mantles.

\begin{theorem} \label{thm:main-thm}
Let $\eta$ be an ordinal and $G$ be generic over $V$ for $\Mbb(\eta)$, the class forcing notion of Definition \ref{def:main-def}. Then in $V[G]$ we have that the $\eta$-th inner mantle $M^\eta$ is $V$. 
\end{theorem}

\begin{proof} For notational convenience, set $\Pbb = \Mbb(\eta)$. For $i < \eta$ let $\Pbb_i = \Pbb \rest R_{\ge i}$. 
We can identify $\Pbb_i$ with its canonical embedding into $\Pbb$, namely by putting $\one$ in each new coordinate. Under this identification
\[
\Pbb = \Pbb_0 \supseteq \Pbb_1 \supseteq \cdots \supseteq \Pbb_i \supseteq \cdots \qquad i < \eta
\]
is a continuous \emph{descending} chain of class forcing notions.%, that is, an iteration of order type $\eta^\star$.

\begin{claim} \label{claim:eta-star}
These inclusions are inclusions of complete suborders. 
% NOTE: changed 'subposets' to 'suborders' since the P_i are not sets, they are classes.
\end{claim}

Let $i < j < \eta$. We want to see that every condition in $\Pbb_i$ has a reduction in $\Pbb_j$. This is straightforward: If $p \in \Pbb_i$ then consider $p \rest R_{\ge j} \in \Pbb_j$. Then any $q \le p \rest R_{\ge j}$ in $\Pbb_j$ is compatible with $p$ in $\Pbb_i$ as witnessed by strengthening $q$ on the new coordinates to agree with $p$.

It is in this sense that we can think of $\Pbb$ as an iteration along the non-well-founded order $\eta^\star$. The literature does not contain a general theory of iterations along arbitrary partial orders, although a number of specific instances have been investigated  \cite{GroszekJech1991:GeneralizedIterationofForcing, KellnerShelah2011:Saccharinity}.  Kanovei highlights the basic problem of defining iterations along non-well-founded orders, noting that the usual inductive definition fails in this context \cite{Kanovei1999:OnNonWellfoundedIterationsofthePerfectSetForcing}.  We have overcome this difficulty by defining the $\Pbb_i$ through alternative means, as certain suborders of a well-founded set forcing iteration of order type $\Ord$.  The important point is that we achieve the required features of an non-well-founded iteration along $\eta^\star$, namely, at successor stages $\Pbb_i \supseteq \Pbb_{i+1}$ we have that $\Pbb_i$ factors as $\Pbb_{i+1} * \dot \Qbb$ for an appropriate $\Qbb$. (Limit stages are handled separately; see Lemma~\ref{lem:limit-lemma-Vkappa-version} below.)

Working in $V[G]$, we will show that the mantle sequence up to and including $\eta$ is given exactly by the sequence of models $V[G_i]$, where $G_i=G\cap \Pbb_i$.  More precisely, 

\begin{claim}
For each $i \le \eta$ we have $(M^i)^{V[G]} = V[G_i]$.
\end{claim}

In particular, this claim immediately implies that $(M^i)^{V[G]}$, for $i \le \eta$, is a model of $\ZFC$, since $V[G_i]$ is an extension of $V$ by a progressively distributive iteration. It is not \textit{prima facie} clear that $M^i$ for limit $i$ should satisfy choice, and the definition of the mantle needs the axiom of choice. So as a consequence, this claim implies that it is always sensical to ask about the mantle of $M^i$ in this context.

Suppose inductively the claim is true for some $i < \eta$, so $(M^i)^{V[G]} = V[G_i]$.  We must show that the mantle of $V[G_i]$ is equal to $V[G_{i+1}]$.  
As $V[G_{i+1}]$ is the smaller model, we analyze $V[G_i]$ as a forcing extension of $V[G_{i+1}]$.  Consider the collection $\dot{\Qbb} = \{p\rest R_i : p\in\Pbb \}$.  It follows from the definition that  $\dot{\Qbb}$ is (equivalent to) a $\Pbb_{i+1}$-name for a partial order, and so $\Pbb_i = \Pbb_{i+1}*\dot{\Qbb}$.  Working in $V[G_{i+1}]$ and taking $\Qbb$ to be the partial order obtained by valuating $\dot{\Qbb}$ with $G_{i+1}$---technically, $\Qbb$ consists of functions  $q_p = \{\langle\alpha,p(\alpha)^{G_{i+1}}\rangle : \alpha \in \dom(p)\}$ for each $p\in \Pbb$---we see immediately that $\Qbb$ is a set-support product of partial orders over the class of cardinals $R_i$.  From the analysis of stage $\alpha$ forcing above, we see that $\Qbb$ is exactly the product 
\[
\Qbb = \Pi_{\alpha\in R_i} \left(\Add(\alpha,(2^{<\alpha})^{++}) \oplus \Add(\alpha^+,1)\right)
\]
as defined in $V[G_{i+1}]$. It remains to show that this forcing makes the mantle of the extension $M^{V[G_i]}$ equal to the ground model $V[G_{i+1}]$.  This argument is presented in detail in \cite[Theorem 66]{FuchsHamkinsReitz2015:Set-theoreticGeology}, and we provide a brief overview here.  

For the forward inclusion $M^{V[G_i]} \subseteq V[G_{i+1}]$, observe that any particular set $X$ added by $\Qbb$ will be added by some initial set-sized factor (by increasing closure of the tails of the product), and by the commutative property of products we might just as well view this factor as occurring after the tail forcing, so $X$ is added to $V[G_i]$ by set forcing over an inner model.  Thus every set added by $\Qbb$ is excluded from a ground, and hence from the mantle.  For the reverse inclusion $V[G_{i+1}] \subseteq M^{V[G_i]}$, it suffices to consider sets $X$ of ordinals in the model $V[G_{i+1}]$.  A density argument shows that $X$ will be coded into the $\GCH$ pattern on a block of cardinals in $R_i$ according to whether the generic opted for $\Add(\alpha,(2^{<\alpha})^{++})$ or $\Add(\alpha^+,1)$ at each stage in the block.  Furthermore, this coding will appear repeated arbitrarily high in the cardinals in $R_i$.  As set forcing cannot alter the $\GCH$ pattern above the size of the forcing, it follows that $X$ will remain coded in any ground model of $V[G_i]$. Thus the mantle of $V[G_i]$ is exactly $V[G_{i+1}]$.

We now turn our attention to the limit case, supposing inductively that the claim holds below some limit ordinal $i\leq\eta$.  Here we will follow the lead of Jech, who proved the following in his work on iterated $\HOD$.

\begin{lemma}\cite{Jech:ForcingWithTreesandOrdinalDefinability}\label{lem:jechs-limit-lemma} 
Suppose $i$ is a limit ordinal and $\Bbb$ is a $\mathord< i^+$-distributive complete Boolean algebra.  Let
\[
\Bbb = \Bbb_0 \supseteq \Bbb_1 \supseteq \cdots \supseteq \Bbb_j \supseteq \cdots \supseteq \Bbb_i \qquad j \leq i
\]
be a continuous descending sequence of complete subalgebras (i.e. $\Bbb_{j+1}$ is a complete suborder of $\Bbb_j$ and for limit $k$ we have $\Bbb_k = \bigcap_{j < k} \Bbb_j$) and let $G$ be $V$-generic for $\Bbb$, with $G_j = G\cap\Bbb_j$ for each $j \leq i$.  If $X$ is a set of ordinals and $X\in V[G_j]$ for all $j \leq i$, then $X \in V[G_i]$.
\end{lemma}

The proof relies heavily on the properties of complete Boolean algebras, and unfortunately does not transfer directly to the class forcing described in Definition \ref{def:main-def}.  The main obstruction is that proper class partial orders do not in general have Boolean completions \cite{HKS2018}.
However, the result can still be transferred to certain proper class partial orders, provided we impose some additional factoring conditions.

\begin{lemma}[Over G\"odel--Bernays set theory without global choice]\label{lem:limit-lemma-Vkappa-version}
Let $i$ be a limit ordinal. 
Suppose that $\Pbb$ is a $\mathord<i^+$-closed pretame class forcing notion and that
\[
\Pbb = \Pbb_0 \supseteq \Pbb_1 \supseteq \Pbb_2 \supseteq \cdots \supseteq \Pbb_j  \supseteq \cdots \supseteq \Pbb_i \qquad j \le i
\]
is a continuous descending sequence of complete suborders which is coded as a single class. Suppose further that $\Pbb$  is a progressively distributive iteration, so for arbitrarily large $\kappa$ we have $\Pbb = \Qbb_\kappa * \Qbb^\tail$ witnessing progressive distributivity. Finally, suppose that  $\Pbb_j \cap \Qbb_\kappa$ is a complete suborder of $\Pbb_j$ for each $j$, and that the intersections form a continuous descending sequence of complete suborders:
\[
(\Pbb\cap\Qbb_\kappa) = (\Pbb_0\cap\Qbb_\kappa) \supseteq (\Pbb_1\cap\Qbb_\kappa)  \supseteq \cdots \supseteq (\Pbb_j\cap\Qbb_\kappa)  \supseteq \cdots \supseteq (\Pbb_i\cap\Qbb_\kappa) \qquad j \le i.
\]
If $G \subseteq \Pbb$ is generic over the ground universe and $G_j = G \cap \Pbb_j$, then for any set of ordinals $X \in V[G]$ we have $X \in \bigcap_{j < i} V[G_j]$ if and only if $X \in V[G_i]$. 
\end{lemma}

\begin{proof}
Fix $X$ a set of ordinals in $V[G]$ and $\kappa > \rank{X}$ as in the lemma.  As $\Qbb_\kappa \forces \Qbb^\tail$ is $\leq\kappa$-distributive, we conclude that $V_\kappa^{V[G]}=V_\kappa^{V[G\cap\Qbb_\kappa]}$, so it suffices to show that $X \in \bigcap_{j < i} V[G_j \cap \Qbb_\kappa]$ iff $X \in V[G_i \cap \Qbb_\kappa]$.  But this follows immediately by applying Lemma \ref{lem:jechs-limit-lemma} to the sequence
\[
(\Pbb\cap\Qbb_\kappa) = (\Pbb_0\cap\Qbb_\kappa) \supseteq (\Pbb_1\cap\Qbb_\kappa)  \supseteq \cdots \supseteq (\Pbb_j\cap\Qbb_\kappa)  \supseteq \cdots \supseteq (\Pbb_i \cap \Qbb_\kappa) \qquad j \le i.
\]
More properly, the lemma is applied to the sequence of Boolean completions of these posets.
\end{proof}

For the limit case of the claim, we would like to apply Lemma \ref{lem:limit-lemma-Vkappa-version} to the sequence $\seq{\Pbb_j : j \leq i}$.  Note that this sequence is a continuous descending sequence of complete suborders and is $\mathord<i^+$ closed, as required.  To see that it is  progressively distributive, fix $\kappa$ and, viewing $\Pbb$ as an $\Ord$-length iteration of set forcing, let $\Qbb_\kappa = \{p\rest(\kappa+1) : p\in\Pbb\}$ be the part of $\Pbb$ lying at and below $\kappa$.  As argued above, the forcing beyond $\kappa$ is $\mathord\leq \kappa$-distributive.  The factoring conditions of Lemma \ref{lem:limit-lemma-Vkappa-version} follow from the observation that, for each $j < i$, we have $(\Pbb_j \cap \Qbb_\kappa) = \Pbb_j \rest (\kappa+1) = \Pbb \rest \left(R_{\ge j} \cap (\kappa+1) \right)$, and the restriction maps give complete embeddings.  Thus $\seq{\Pbb_j : j \leq i}$ satisfies the hypotheses of Lemma \ref{lem:limit-lemma-Vkappa-version}. And so we can conclude that $\bigcap_{j < i} V[G_j] = V[G_i]$.

Finally, observe  that in the case when $i=\eta$ we have $\bigcap_{j<\eta} \Pbb_j$ is trivial forcing, and so $(M^\eta)^{V[G]} = V[G_\eta] = V$ as desired.  This completes the proof of the claim, and of Theorem \ref{thm:main-thm}.
\end{proof} 

We now analyze the iterated $\HOD$ sequence\footnote{See Definition \ref{definition.iteratedHOD}.} 
in the extension by $\Mbb(\eta)$.

\begin{theorem} \label{thm:meta-iterated-hods}
Let $\eta$ be an ordinal and $G$ be generic over $V$ for $\Mbb(\eta)$, the class forcing notion of Definition \ref{def:main-def}. Then in $V[G]$ we have that the $\eta$-th iterated $\HOD$ is  $\HOD^\eta = V$. 
\end{theorem}

\begin{proof}
We follow the proof of Theorem \ref{thm:main-thm}, defining for each $i<
\eta$ the partial order $\Pbb_i$ with corresponding generic $G_i$. 

\begin{claim}
Let $i \le \eta$. In $V[G]$ the $i$th iterated $\HOD$ is $\HOD^i = V[G_i]$
\end{claim}

This is proven inductively. The limit stage of the argument is the same as the limit stage of the argument for the sequence of inner mantles. For both sequences, the elements of the sequence at limit indices are defined as the intersection of the previous elements. So the same argument using Jech's lemma goes through in this case.

It remains to verify the successor stage of the induction. We follow the presentation in \cite[Theorem 66]{FuchsHamkinsReitz2015:Set-theoreticGeology} to show that $\HOD^{V[G_i]} = V[G_{i+1}]$.  By the proof of Theorem \ref{thm:main-thm}, $V[G_i]$ is a forcing extension of $V[G_{i+1}]$ by the product forcing
\[
\Qbb = \prod_{\alpha\in R_i} \left(\Add(\alpha,(2^{<\alpha})^{++}) \oplus \Add(\alpha^+,1)\right)
\]
and every set of ordinals $X$ in $V[G_{i+1}]$ is coded into the $\GCH$ pattern in $R_i$.  Furthermore, the coding is such that $X$ can be defined in $V[G_i]$ from $R_i$ together with the interval on which the coding takes place and, as $R_i$ remains ordinal definable in $V[G_i]$, $X$ is thus ordinal definable there.  It follows that in $V[G_i]$ we have that $\HOD \supseteq V[G_{i+1}]$.

We now want to see the other containment.  As $\Qbb$ is a progressively-closed product, any new set $X$ added by $\Qbb$ must be added by an initial segment $\Qbb_\alpha = \Qbb \rest \alpha$.  Factoring $\Qbb=\Qbb_\alpha \times \Qbb^\tail =\Qbb^\tail \times \Qbb_\alpha$ we see that $X$ is added to the model $V[G_i]$ by the set forcing $\Qbb_\alpha$.  We would like to use the standard fact that sets added by weakly homogeneous forcing cannot be ordinal definable in the extension.\footnote{Recall that a forcing $\Pbb$ is \emph{weakly homogeneous} if given any $p,q \in \Pbb$ there is an automorphism $\pi$ of $\Pbb$ so that $\pi(p)$ and $q$ are compatible.}
Unfortunately, the forcing $\Qbb$ is not weakly homogeneous, because $\Add(\xi,(2^{<\xi})^{++})$ is not isomorphic to $\Add(\xi^+,1)$.  However, $\Qbb_\alpha$ is \emph{densely weakly homogeneous}: there is a dense set of conditions $q$ so that $\Qbb_\alpha \rest q$ is weakly homogeneous.  These $q$ are exactly those conditions that make a definite choice in the lottery at each stage below $\alpha$ (recall that $\Qbb$ has full support, not Easton support, which allows conditions with support $\alpha$).  Below such a condition, $\Qbb_\alpha$ is simply a product of Cohen forcing and thus weakly homogeneous.  Therefore $X$ is added to $V[G_i]$ by weakly homogeneous forcing, namely $\Qbb_\alpha \rest q$.  

Thus $X$ is not ordinal definable in $V[G_i]$, and so $\HOD^{V[G_i]} \subseteq V[G_{i+1}]$.  This establishes the other direction of the inclusion, completing the proof of the claim, and hence the theorem.
\end{proof}

Combined with Theorem \ref{thm:main-thm} this completes the proof of the first main theorem. \smallskip

The technique above suffices to generate mantle sequences of length $\eta$ for any ordinal. But we needed the forcing to be $\mathord<\eta^+$-closed to make the limit step of the argument go through. It is natural to ask whether we can make the mantle sequence of have length $\Ord$---can we force $V=M^\Ord$ of the extension? (The analogous question for iterated $\HOD$ was already answered by Zadro\.zny in the positive \cite{Zadrozny83:IteratingOrdinalDefinability}.) And what about well orders of length greater than $\Ord$?  There are many such orders already definable in $\ZFC$ (such as $\Ord+1$, $\Ord+\Ord$, $\Ord \times \Ord$, and so on), and various second-order theories guarantee the existence of even longer well orders.

\begin{conjecture}
For any class well-order $\Gamma$ there is a $\ZFC$-preserving class forcing notion $\Mbb(\Gamma)$ definable from $\Gamma$ such that forcing with $\Mbb(\Gamma)$ yields an extension in which the mantle sequence does not stabilize before $\Gamma$ and the $\Gamma$-th inner mantle is the ground model, $M^\Gamma = V$.
\end{conjecture}

The reader may worry about iterating the mantle beyond $\Ord$, that something goes wrong with non-set-like well-orders. To assuage that worry, let us show that sufficiently strong axioms imply the $\Gamma$-th mantle always exists. Elementary Transfinite Recursion $\ETR$, originally introduced by Fujimoto \cite{fujimoto2012}, asserts that transfinite recursions of first-order properties along well-founded class relations have solutions.\footnote{The reader who is familiar with the reverse mathematics of second-order arithmetic should compare $\ETR$ to arithmetical transfinite recursion, the analogous principle in arithmetic.}
See \cite{gitman-hamkins2017} for a formal definition and further discussion. This principle is strictly stronger than G\"odel--Bernays set theory, implying $\Con(\ZFC)$ and more. On the other hand, it is weaker than Kelley--Morse set theory, full impredicative second-order set theory, and indeed weaker than $\Pi^1_1$-Comprehension.

\begin{proposition}[Over G\"odel--Bernays set theory without global choice]
Assume $\ETR$. Then $M^\eta$ exists for every ordinal $\eta$ and $M^\Gamma$ exists for every class well-order $\Gamma$.
\end{proposition}

\begin{proof}
The point is that the construction of the sequence of inner mantles along $\Gamma$ is an elementary recursion along $\Gamma$. So $\ETR$ says a solution to this recursion exists.
\end{proof}

So it is sensible to iterate the mantle along any class well-order, even those which are not set-like. 

And seeing that we can iterate the mantle beyond $\Ord$, we can ask whether it ever stabilizes at some class well-order.
%% before resubmitting, make sure [page 3, line 2] response has the correct numbering for this question, in case things change --KW

\begin{question}
Is there a model of $\ETR$ in which the sequence of inner mantles never stabilizes? That is, is there a model of $\ETR$ in which for every class well-order $\Gamma$ we have that the inner mantle sequence along $\Gamma$ does not stabilize?
\end{question}

\section{\texorpdfstring{Different lengths for the sequences of iterated $\HOD$s and inner mantles}{Different lengths for the sequences of iterated HODs and inner mantles}} \label{sec:different-lengths} %pdf metadata doesn't like the mathmode fonts :(

Combining information from the proofs of Theorems \ref{thm:main-thm} and \ref{thm:meta-iterated-hods}, which together comprise the first main theorem, we get models where the sequences of iterated $\HOD$s and inner mantles both stabilize at or after $\eta$, and where $M^i = \HOD^i$ for all $i \le \eta$. It is natural to ask whether we can separate the sequences. As a first question: Can they have different lengths? It follows immediately from the first main theorem plus the work of \cite{FuchsHamkinsReitz2015:Set-theoreticGeology} that the answer is yes: force with one of their forcings which separate the mantle and the $\HOD$, followed by $\Mbb(\eta)$.
This then gives a model where the sequence of inner mantles has length at least $\eta + 1$ and the sequence of iterated $\HOD$s has length $\eta$, or vice versa if we separated the other way. One can see that several other special cases also follow immediately.
But we would like to do better than that, and independently control the lengths of the two sequences to be any two ordinals we please.

It is the content of our second main theorem that we can in fact do better. Let us briefly sketch the strategy for proving such before going into detail. For the sketch let us only consider the case where the sequence of inner mantles is longer than the sequence of iterated $\HOD$s, as the other case is analogous. We first want to force to get a long sequence of inner mantles, but have $V = \HOD$ in the extension. We then force with $\Mbb(\eta)$, which we have already analyzed, to get a model where the sequence of inner mantles is longer than the iterated $\HOD$ sequence. If we preceded both forcings by first forcing to ensure $V = \HOD = M$, we could exactly control the lengths of the sequences. 

So the missing work is to get a forcing, call it $\Nbb(\eta)$, which will force the ground model to be the $\eta$-th inner mantle while forcing the extension to be its own $\HOD$. And to separate the lengths in the other direction we will need another forcing, call it $\Obb(\eta)$, which makes the ground model the $\eta$-th iterated $\HOD$ but where the extension is its own mantle.

Let us start by considering $\Nbb(\eta)$. We begin by describing the self-encoding forcing, which will be a basic element for building $\Nbb(\eta)$. 

\begin{definition}
Let $\alpha$ be a cardinal. The \emph{self-encoding forcing at $\alpha$}, call it $\Sbb_\alpha$, is an iteration of length $\omega$ which affects the $\GCH$ pattern on the interval $I_\alpha = [\alpha,\lambda_\alpha)$, where $\lambda_\alpha$ is the least beth fixed point $>\alpha$. Stage $0$ of the iteration forces with $\Add(\alpha,1)$ to produce a generic $g_0 \subseteq \alpha_0 = \alpha$. Stage $1$ then codes $g_0$ into the $\GCH$ pattern past $\alpha_0$.  In  \cite{FuchsHamkinsReitz2015:Set-theoreticGeology} this is accomplished by forcing twice: first, the canonical forcing of the $\GCH$ on the interval (adding a subset to each regular cardinal) is used to provide a `clean slate' while preserving the beth fixed point $\lambda_0$, and second, an Easton support product is used to force the $\GCH$ to hold or fail at each cardinal in turn, coding $g_0$ into the resulting pattern.  Note that the cardinals on $I_\alpha$ that are collapsed by the canonical forcing of the $\GCH$ are entirely determined in the ground model and thus constitute a definable set there. With this in mind, it is a matter of careful bookkeeping to interleave these two forcings into a single Easton support product, call it $\Sbb_\alpha^1$. Namely, we will do an Easton support product on the regular cardinals in the interval $[\alpha_0, \beth_{\alpha_0}(\alpha_0))$ so that at stage $\gamma$ we force with one of the following, where the definitions take place in $V[g_0]$:

\begin{itemize}
\item $\Add(\gamma,1)$ if $\gamma$ is collapsed by the canonical forcing of the $\GCH$,
; otherwise
\item $\Add(\gamma,1)$ if the $\GCH$ holds at $\gamma$ and we want to force $\GCH$ to hold at $\gamma$ (this ensures all necessary cardinals are collapsed);
\item $\Add(\gamma^+,1)$ if $\GCH$ fails at $\gamma$ and we want to force $\GCH$ to hold at $\gamma$.
\item $\Add(\gamma,\delta)$, where $\delta$ is the double successor of $\gamma$ among those cardinals that survive the canonical forcing of the $\GCH$, if we want to force $\GCH$ to fail at $\gamma$;
\end{itemize}

In this way, after stage $1$ of the forcing $g_0$ is coded into the $\GCH$ pattern on the $\alpha_0$ many regular cardinals following $\alpha_0$. Let $g_1 \subseteq \Sbb_\alpha^1$ be the generic (over $V[g_0]$) from stage $1$ of the forcing. By means of a pairing function on the ordinals we may consider $g_1$ as a subset of $\alpha_1 = \beth_{\alpha_0}(\alpha_0)$. Note that $\alpha_1$ is the supremum of the $\alpha_0$ many surviving cardinals past $\alpha_0$, as we cofinally often collapsed $2^\gamma$ to be $\gamma^{++}$. 

We then proceed inductively to define the further stages. That is, at stage $n+1$ we want to code the generic $g_n \subseteq \alpha_{n}$ from the previous stage into the $\GCH$ pattern on the regular cardinals in the interval $[\alpha_{n},\beth_{\alpha_n}(\alpha_n))$. This is done by an Easton support product $\Sbb_\alpha^{n+1}$, defined in $V[g_0 * \cdots * g_n]$, where at stage $\gamma$ one of $\Add(\gamma,1)$, $\Add(\gamma^+,1)$, or $\Add(\gamma,\delta)$ is chosen, as in the case for stage $1$. Then if $g_{n+1} \subseteq \Sbb_\alpha^{n+1}$ is the generic (over $V[g_0 * \cdots * g_n]$) for this product we can again consider it as a subset of $\alpha_{n+1} = \beth_{\alpha_0}(\alpha_0)$. Similar to stage $1$, we get that $\alpha_{n+1}$ is the supremum of the $\alpha_n$ many surviving cardinals.
\end{definition}

Observe that the supremum of the $\alpha_n$ is the least beth fixed point above $\alpha$, which we called $\lambda_\alpha$. So we can calculate that $\Sbb_\alpha$ has cardinality $\lambda_\alpha^{\ \omega}$. So forcing with $\Sbb_\alpha$ does not affect the $\GCH$ pattern outside of an interval of the form $(\gamma,\lambda_\alpha^{\ \omega})$, where $2^\gamma \le \alpha$. Also observe that $\lambda_\alpha$ is still the least beth fixed point $> \alpha$ after forcing with $\Sbb_\alpha$. This yields that $\Sbb_\alpha$ preserves the class of beth fixed points.

For defining $\Nbb(\eta)$ we will need to use different coding points than were used for $\Mbb(\eta)$. Fix $\eta$. Let $R$ be the class of cardinals $\alpha > \eta^+$ of the form $\alpha = (2^\lambda)^+$ where $\lambda$ is a beth fixed point. As before, we partition $R$ into $R_i$, for $i < \eta$, consisting of the cardinals $\alpha$ whose index $i(\alpha)$ in $R$ is equivalent to $i$ modulo $\eta$. And we define $R_{>i} = \bigcup_{j > i} R_j$ and $R_{\ge i} = \bigcup_{j \ge i} R_j$ as we did earlier.

\begin{definition}
Fix an ordinal $\eta$ and let $R$ be as just defined. The forcing $\Nbb(\eta)$ is the class forcing whose conditions are set-sized functions $p$ with domain an initial segment of $R$ such that for each $\alpha \in \dom p$ we have that $p(\alpha)$ is an $\Nbb(\eta) \rest (R_{>i(\alpha)} \cap \alpha)$-name for a condition in $\Sbb_\alpha$. Given $p,q \in \Nbb(\eta)$ say that $q \le p$ if $\dom(q) \supseteq \dom(p)$ and for all $\alpha \in \dom(p)$ we have that $p \rest (R_{>i(\alpha)} \cap \alpha)$ forces over $\Nbb(\eta) \rest (R_{>i(\alpha)} \cap \alpha)$ that $q(\alpha) \le p(\alpha)$. 
\end{definition}

Let us check some basic properties of $\Nbb(\eta)$.

\begin{lemma}
Fix an ordinal $\eta$. The forcing $\Nbb(\eta)$ has the following properties.
\begin{enumerate}
\item $\Nbb(\eta)$ is $\mathord<\eta^+$-closed.
\item Forcing with $\Nbb(\eta)$ preserves $R$ and $R_i$ for $i < \eta$.
\item $\Nbb(\eta)$ is a progressively distributive iteration.
\item And the same are true of $\Nbb(\eta) \rest R_{\ge i}$ for $i \le \eta$.
\end{enumerate}
\end{lemma}

\begin{proof}
$(1)$ Because each stage in the forcing is $\eta^+$-closed, which is because we only use coding points $>\eta^+$. 

$(2)$ Because of how we chose the coding points. %%more detail? --KW

$(3)$ This is an application of Theorem \ref{thm.generalizedcoheniterations}. The point is, $\Nbb(\eta)$ is an iteration of iterations of Cohen forcings defined in appropriate inner models. But this can be thought of as a single iteration, and so $\Nbb(\eta)$ is a generalized Cohen iteration, using that the coding points are spaced out sufficiently to preserve the necessary cardinals.

$(4)$ By the same arguments.
\end{proof}

\begin{theorem}
Fix $\eta$. Let $G \subseteq \Nbb(\eta)$ be generic over $V$. Then, $(M^\eta)^{V[G]} = V$ and $\HOD^{V[G]} = V[G]$.
\end{theorem}

\begin{proof}
Once again for notational convenience let $\Pbb = \Nbb(\eta)$ and set $\Pbb_i = \Pbb \rest R_{\ge i}$, to obtain
\[
\Pbb = \Pbb_0 \supseteq \Pbb_1 \supseteq \cdots \supseteq \Pbb_i \supseteq \cdots \qquad i < \eta
\]
a continuous descending chain of complete suborders. As before, we can think of it as an iteration of order type $\eta^\star$; see the discussion surrounding Claim~\ref{claim:eta-star}. Set $G_i = G \cap \Pbb_i$. Work in $V[G_{i+1}]$. Here $V[G_i]$ is a forcing extension of $V[G_{i+1}]$ by a forcing $\Qbb$ which is the progressively closed product
\[
\Qbb = \prod_{\alpha \in R_i} \Sbb_\alpha,
\]
similar to the analysis of $\Mbb(\eta)$. 

Let us now verify the claim about the $\eta$-th mantle of $V[G]$. We will establish the following.

\begin{claim}
For each $i \le \eta$ we have $(M^i)^{V[G]} = V[G_i]$.
\end{claim}

Again, this claim immediately implies that $(M^i)^{V[G]}$ satisfies $\ZFC$ for each $i \le \eta$.

Suppose inductively the claim is true for some $i < \eta$, that is $(M^i)^{V[G]} = V[G_i]$. We want to show that the mantle of $V[G_i]$ is $V[G_{i+1}]$. Our argument here follows that of \cite[Theorem 67]{FuchsHamkinsReitz2015:Set-theoreticGeology}. 
Factor $\Pbb_i$ as $\Pbb_{i+1} * \dot \Qbb$ and factor $G_i$ as $G_{i+1} * H$ where $H \subseteq \Qbb$ is generic over $V[G_{i+1}]$. Let $\Qbb_\alpha = \prod_{\beta \in R_i \cap \alpha} \Sbb_\beta$ be the initial segment of $\Qbb$ below $\alpha$ and $\Qbb^\alpha$ be the tail of $\Qbb$ beyond $\alpha$, so that $\Qbb$ factors as $\Qbb_\alpha \times \Sbb_\alpha \times \Qbb^\alpha$. Then $\Qbb_\alpha$ has size $<\alpha$, due to the spacing of the coding points, and $\Qbb^\alpha$ is $\mathord\le\lambda_\alpha$-closed. So in $V[G_i] = V[G_{i+1}][H]$ the behavior of the $\GCH$ pattern on the interval $I_\alpha = [\alpha,\lambda_\alpha)$, where $\lambda_\alpha$ is the least beth fixed point $>\alpha$, is determined entirely by $\Sbb_\alpha$. 
Let $H_\alpha = H \cap \Qbb_\alpha$ and $H^\alpha = H \cap \Qbb^\alpha$. Because $\Qbb$ is a progressively closed product we have that any set in $V[G_{i+1}][H]$ is already in $V[G_{i+1}][H_\alpha]$ for some $\alpha$. 

Let us see that $V[G_{i+1}] \subseteq M^{V[G_{i+1}][H]}$. Any ground of $V[G_{i+1}][H]$ must agree with it about the $\GCH$ pattern on a tail. And by a density argument every set of ordinals in $V[G_{i+1}]$ was coded into the $\GCH$ pattern of $V[G_{i+1}][H]$ cofinally often. So $V[G_{i+1}]$ is contained in every ground of $V[G_{i+1}][H]$, and thus is contained in the mantle. For the other direction of the containment, observe that $V[G_{i+1}][H^\alpha]$ is a ground of $V[G_{i+1}][H]$, as $\Qbb_\alpha \times \Sbb_\alpha$ is set-sized forcing. But by increasing closure of the tail forcings, we have that $M^{V[G_{i+1}][H]} \subseteq \bigcap_{\alpha \in R_i} V[G_{i+1}][H^\alpha] = V[G_{i+1}]$, as desired.

We now want to check the limit case of the induction. Fix $i < \eta$ a limit ordinal. The sequence
\[
\Pbb = \Pbb_0 \supseteq \Pbb_1 \supseteq \cdots \supseteq \Pbb_j \supseteq \cdots \qquad j < i
\]
is a continuous descending chain of complete suborders. And $\Pbb$ is $\mathord<i^+$-closed, because it is $\mathord<\eta^+$-closed, and a progressively distributive iteration. Moreover, we get the factoring properties of Lemma \ref{lem:limit-lemma-Vkappa-version}, by a similar argument as in the $\Mbb(\eta)$ case. So we can appeal to that lemma to conclude that $V[G_i] = \bigcap_{j < i} V[G_j]$, establishing the limit case of the induction.

Finally, observe that we have $\bigcap_{j < \eta} \Pbb_j$ is trivial forcing, and so $(M^\eta)^{V[G]} = V[G_\eta] = V$, as desired. So we have seen that $\Pbb = \Nbb(\eta)$ has the desired properties with regard to the sequence of inner mantles.

Now let us see that $\HOD^{V[G]} = V[G]$. As a first step, let us observe that every set in $V[G_1]$ is ordinal definable in $V[G]$. This is because $\Pbb$ can be factored as $\Pbb_1 * \dot\Qbb$ where $\Qbb = \prod_{\alpha \in R_1} \Sbb_\alpha$ and by genericity $\Qbb$ codes every ground model---which in this context is $V[G_1]$---set into the $\GCH$ pattern. In particular, every $\Qbb$-name in $V[G_1]$ is ordinal definable in $V[G]$. And because $\Qbb$ is a progressively closed iteration we get that every set in $V[G]$ already appears in $V[G_1][H_\alpha]$ for some $\alpha$, where $H_\alpha = H \cap \Qbb_\alpha$ and $H$ comes from factoring $G$ as $G_1 * H$. In other words, every set in $V[G]$ is of the form $\tau_{H_\alpha}$, where $\tau \in V[G_1]$ is a $\Qbb_\alpha$-name for some $\alpha$. But for each stage $\beta$ of $\Qbb_\alpha$ the generic filter added at that stage was coded into the $\GCH$ pattern on the interval $I_\beta = [\beta,\lambda_\beta)$. And in $V[G]$ we can definably combine these filters together on the product, using that $R$ is preserved, and so $H_\alpha$ is ordinal definable in $V[G]$. So $\tau_{H_\alpha}$ is ordinal definable in $V[G]$, establishing that $\HOD^{V[G]} = V[G]$.
\end{proof}

We now turn to describing $\Obb(\eta)$, the forcing which forces the ground model to be the $\eta$-th iterated $\HOD$ of the extension while the extension is its own mantle. We follow \cite[Theorem 70]{FuchsHamkinsReitz2015:Set-theoreticGeology}, which in turn is an adaptation of the main result of \cite{HamkinsReitzWoodin2008:TheGroundAxiomAndVequalsHOD}. The key insight is that forcing with an appropriately chosen Silver iteration will preserve the $\HOD$ while forcing the Ground Axiom, which asserts that $V = M$, to hold. 

\begin{definition}
Fix an ordinal $\eta$. Let $R$ be a suitably chosen class of coding points above $\eta^+$, as in the discussion from section~\ref{sec:meta}. In particular, $R$ is chosen so that any two elements of $R$ have a strong limit cardinal between them.  And let $\Mbb(\eta)$ be the forcing from Definition~\ref{def:main-def} using $R$ for its coding points. Let $\dot \Sbb$ be an $\Mbb(\eta)$-name for an Easton support Silver iteration, the $\Ord$-length iteration that adds a Cohen subset to the regular cardinals $(2^\alpha)^+$ for $\alpha \in R$. Set $\Obb(\eta) = \Mbb(\eta) * \dot \Sbb$.
\end{definition}

Observe that the stages of the Silver iteration occur between the coding points used by $\Mbb(\eta)$, since $R$ was chosen to be sufficiently spaced out. And forcing with the Silver iteration will not affect whether the $\GCH$ holds on $R$ and preserves the definition of $R$. 

\begin{theorem}
Fix $\eta$. Let $G \subseteq \Obb(\eta)$ be generic over $V$. Then, $(\HOD^\eta)^{V[G]} = V$ and $M^{V[G]} = V[G]$.
\end{theorem}

\begin{proof}
For notational convenience, let $\Pbb$ be an alias for $\Mbb(\eta)$. Factor $G$ as $H * K$ where $H \subseteq \Pbb$ is generic over $V$ and $K \subseteq \Sbb$ is generic over $V[H]$. It follows from Theorem~\ref{thm:main-thm} that in $V[H]$ the $\eta$-th inner mantle and the $\eta$-th iterated $\HOD$ are both $V$. To prove the theorem it therefore suffices establish two claims: first that $\HOD^{V[H][K]} = \HOD^{V[H]}$; and second that $M^{V[H][K]} = V[H][K]$. 

\begin{claim}
$\HOD^{V[H][K]} = \HOD^{V[H]}$.
\end{claim}

The Silver iteration $\Sbb$ is weakly homogeneous and ordinal definable. Thereby we can conclude that $\HOD^{V[H][K]} \subseteq \HOD^{V[H]}$. For the other direction, factor $\Pbb$ as $\Pbb_1 * \dot \Qbb$, where $\Pbb_1 = \Pbb \rest R_{\ge 1}$. Let $H_1 * J$ be the corresponding factoring of $H$. Every set in $V[H_1]$ is coded arbitrarily high in the $\GCH$ pattern on $R$ in $V[H_1][J]$. And since $\Sbb$ does not affect $R$ nor the $\GCH$ pattern on $R$, we get that every set in $V[H_1]$ is coded arbitrarily high in the $\GCH$ pattern on $R$ in $V[H_1][J][K] = V[H][K]$. Thus, $\HOD^{V[H][K]} \supseteq V[H_1] = \HOD^{V[H]}$. This finishes the proof of the claim.

\begin{claim}
$M^{V[H][K]} = V[H][K]$.
\end{claim}

The key idea here is that we can think of forcing with $\Pbb = \Mbb(\eta)$ as having a last step, which is essentially the forcing to make the ground model the $\HOD$ and the mantle of the extension---cf. \cite[Theorem 66]{FuchsHamkinsReitz2015:Set-theoreticGeology}. So we think of $\Pbb * \dot \Sbb$ as $\Pbb_1 * \dot \Qbb * \dot \Sbb$. In their Theorem 70---which is an adaptation the main result of \cite{HamkinsReitzWoodin2008:TheGroundAxiomAndVequalsHOD}---Fuchs, Hamkins, and Reitz show that forcing with $\Sbb$ after forcing with $\Qbb$ forces the Ground Axiom. So the same argument applied here will show that $V[H][K]$ is its own mantle. We sketch the argument here and refer the reader to \cite{FuchsHamkinsReitz2015:Set-theoreticGeology} or \cite{HamkinsReitzWoodin2008:TheGroundAxiomAndVequalsHOD} for full details.

Once again factor $\Pbb$ as $\Pbb_1 * \dot \Qbb$ and $H$ as $H_1 * J$. Suppose towards a contradiction that $V[H][K]$ has a nontrivial ground $W$. That is, $V[H][K] = W[\ell]$ for $\ell \in V[H][K]$ generic over $W$ for some nontrivial set forcing $\Rbb \in W$. Because $W$ and $V[H][K]$ must agree an a tail of the $\GCH$ pattern and because all of $V[H_1]$ is coded into the $\GCH$ pattern in $V[H][K]$, we can conclude that $V[H_1] \subseteq W$. We then factor $J = J_\alpha \times J^\alpha$ and $K = K_\alpha * K^\alpha$ at $\alpha$ sufficiently far above the size of $\Rbb$. By a $\delta$-approximation and covering argument for appropriately chosen $\delta > \card{\Rbb}^W$ we can then see that $V[H_1][J^\alpha] \subseteq W$. 

Next, pick $A$ in $W$ which codes all the subsets of $\Rbb$, $\Rbb$-names which will be interpreted to be the generics $J_\alpha$, and $K_\alpha$, and all subsets of $\delta$. So then $V[H_1][J^\alpha][A][\ell]$ contains $V[H_1][J][K_\alpha]$.  Because the coding points for $\Qbb$ and the levels in $\Sbb$ are sufficiently spaced out, this $A$ can be chosen to be smaller than the closure of the tail of $\Sbb$. So then $A,\ell \in V[H_1][J][K]$ must already be in $V[H_1][J][K_\alpha]$, and so $V[H_1][J][K_\alpha] = V[H_1][J^\alpha][A][\ell]$. Therefore, $V[H_1][J][K]$ is the forcing extension of $V[H_1][J^\alpha][A]$ by $\ell * K^\alpha$, and it is also the forcing extension of $W$ by $\ell$. So by another $\delta$-approximation and covering argument, using that $A$ codes enough information about $\delta$ to ensure that $W$ and $V[H_1][J^\alpha][A]$ agree on the cardinal successor of $\delta$, we can conclude that $W = V[H_1][J^\alpha][A]$. But then 
\[
W[\ell] = V[H_1][J][K_\alpha] \subsetneq V[H_1][J][K] = W[\ell].
\]

Having reached the desired contradiction we conclude $V[H][K]$ must have no nontrivial grounds. Therefore it is its own mantle, completing the proof of the claim, which finishes the proof of the theorem.
\end{proof}

With these forcings in hand, we are now ready to state and prove the second main theorem, that we can force the sequences of inner mantles and of iterated $\HOD$s to have different lengths.

\begin{theorem} \label{thm:independent-lengths}
Let $\zeta$ and $\eta$ be ordinals. Then there are forcings $\Abb$ and $\Bbb$, uniformly definable in $\zeta$ and $\eta$ as parameters, 
so that:
\begin{itemize}
\item Forcing with $\Abb$ gives a model where the sequence of iterated $\HOD$s has length exactly $\zeta$ and the sequence of inner mantles has length exactly $\zeta + \eta$.
\item Forcing with $\Bbb$ gives a model where the sequence of inner mantles has length exactly $\zeta$ and the sequence of iterated $\HOD$s has length exactly $\zeta + \eta$.
\end{itemize}
\end{theorem}

\begin{proof}
Fix ordinals $\zeta$ and $\eta$. Let $\Cbb$ be the class forcing which forces every set to be coded in the $\GCH$ pattern cofinally often---see \cite{Reitz2006:Dissertation}. 

Set $\Abb = \Cbb * \dot \Nbb(\eta) * \dot \Mbb(\zeta)$. Let $G * H * K$ be generic over $V$ for $\Abb$. By the properties of $\Cbb$ we get that $M^{V[G]} = \HOD^{V[G]} = V[G]$. Thus, in $V[G][H]$ the sequence of inner mantles has length exactly $\eta$, while $\HOD^{V[G][H]} = V[G][H]$. Finally, we get that in $V[G][H][K]$ the sequence of inner mantles has length exactly $\zeta + \eta$, while the sequence of iterated $\HOD$s has length exactly $\zeta$, as desired.

Now set $\Bbb = \Cbb * \dot \Obb(\eta) * \dot \Mbb(\zeta)$. A similar analysis shows that in a forcing extension by $\Bbb$ the sequence of inner mantles has length exactly $\zeta$ while the sequence of iterated $\HOD$s has length exactly $\zeta + \eta$.
\end{proof}

To get that the sequence of inner mantles and the sequence of iterated $\HOD$s can be forced to have different lengths, we used forcings that make one sequence an initial segment of the other. How independent can we make the sequence of inner mantles and the sequence of iterated $\HOD$s? Let us ask this question for two specific cases, though many other variants can be asked. The first case is whether we can have the sequences only line up at the beginning, where $M^0 = \HOD^0 = V$, and the end.

\begin{question}
Let $\eta$ be an ordinal. Is there a class forcing which forces the ground model to be both the $\eta$-th inner mantle and the $\eta$-th iterated $\HOD$ of the extension, but for all $i < \eta$ we have $M^i \ne \HOD^i$? Can we moreover get that for all $0 < i,j < \eta$ that $M^i \ne \HOD^j$?
\end{question}

The second case is whether one sequence can ``leapfrog'' over the other.

\begin{question}
Let $\eta$ be an ordinal. Is there a class forcing which forces the sequence of inner mantles to have length $\eta$ and forces that for all $i < \eta$ that $M^i = \HOD^{2i}$? Is there a class forcing which forces the iterated $\HOD$ sequence to have length $\eta$ and forces that for all $i < \eta$ that $\HOD^i = M^{2i}$?
\end{question}

\printbibliography

\end{document}